\documentclass[microtype]{gtpart}     
\usepackage{graphicx}  

\title{$\ \ $ A realizability criterion for rational maps with three branch points}

\author{Zhiqiang Wei}
\address{School of Mathematics and Statistics, Henan University, Kaifeng, China \newline
Center for Applied Mathematics of Henan Province, Henan University, Zhengzhou, China}
\email{weizhiqiang15@mails.ucas.edu.cn  ~~or~~10100123@vip.henu.edu.cn}


\arxivreference{}

\newtheorem{thm}{Theorem}[section]    
\theoremstyle{definition}
\newtheorem{defn}[thm]{Definition}    
\newtheorem*{rem}{Remark}             
\newtheorem{cor}{Corollary}[section]

\begin{document}

\begin{abstract}
Let
\[
\pi_{0}=[\underbrace{2,\ldots,2}_{2s-1}],\quad
\pi_{\infty}=[\underbrace{3,\ldots,3}_{s},\underbrace{1,\ldots,1}_{s-2}],\quad
\pi_{1}=[\gamma_{1},\gamma_{2},\gamma_{3}]
\]
be three nontrivial partitions of the integer \(d=4s-2\), where \(s\ge 3\). We establish a necessary and sufficient condition for the candidate datum \(\{\pi_{0},\pi_{\infty},\pi_{1}\}\) to be realizable by a rational map.
\end{abstract}

\maketitle


\section{Introduction}

Let $M$ and $N$ be closed surfaces with Euler characteristics $\chi(M)$ and $\chi(N)$, respectively. A smooth map $f: M \to N$ is a degree-$d$ \emph{branched covering} if for each $x \in N$ there exists a partition $\pi(x) = [\alpha_1, \ldots, \alpha_r]$ of $d$ (square brackets denote an unordered multiset) such that, in a neighbourhood of $x$ in $N$, $f$ is locally modelled by the map
\[
\widetilde{f}: \{1,\ldots,r\} \times \mathbb{D} \to \mathbb{D}, \qquad \widetilde{f}(j,z) = z^{\alpha_j},
\]
with $x$ corresponding to $0 \in \mathbb{D} = \{z : |z| < 1\} \subset \mathbb{C}$. We denote the length of a partition $\pi(x)$ by $|\pi(x)|$. The points $x \in N$ for which $\pi(x)$ is nontrivial form the finite \emph{branch set} $B_f$ of $f$. The collection $\mathcal{D} = \{\pi(x) : x \in B_f\}$ (with repetitions allowed) is called the \emph{branch datum} of $f$.

It is well known that the degree $d$, the Euler characteristics $\chi(M)$ and $\chi(N)$, and the branch datum $\mathcal D$ must satisfy the Riemann--Hurwitz formula
\begin{equation}\label{RHF}
\nu(\mathcal{D}) = d \cdot \chi(N) - \chi(M),
\end{equation}
where $\nu(\mathcal{D})$ denotes the total branching of $f$. More precisely, if
\[
\mathcal{D} = \{[\alpha_{11},\ldots,\alpha_{1r_1}], \ldots, [\alpha_{n1},\ldots,\alpha_{nr_n}]\},
\]
then
\[
\nu(\mathcal{D}) = \sum_{k=1}^{n} \sum_{j=1}^{r_k} (\alpha_{kj} - 1).
\]

Given closed surfaces $M$ and $N$, a pair $(d, \mathcal D)$ consisting of an integer $d \geq 2$ and a collection $\mathcal D$ of nontrivial partitions of $d$ is called a \emph{candidate branch datum} if it satisfies the Riemann--Hurwitz formula. When $d$ is understood, we sometimes refer to $\mathcal D$ itself as a candidate branch datum.

Given closed surfaces $M$ and $N$ and a collection $\mathcal D$ of partitions of $d$, the problem of determining whether there exists a degree-$d$ branched covering $f: M \to N$ with branch datum $\mathcal D$ is known as the \emph{Hurwitz existence problem}. This is a fundamental question in complex analysis, algebraic topology and algebraic geometry, and it has attracted considerable research interest over the years.

In his classical work, Hurwitz \cite{Hur91} reduced this problem to the realizability of certain cycle types by permutations in symmetric groups. Edmonds, Kulkarni, and Stong \cite{EKS84} proved that every candidate branch datum is realizable when $\chi(N) \leq 0$ (see also \cite{EZ78,Hus62}). However, the case $N = S^2$ is considerably more subtle: there exist candidate branch data that cannot be realized by any branched covering; such data are called \emph{exceptional}. A classical example is $d = 4$ and $\mathcal{D} = \{[3,1], [2,2], [2,2]\}$. Characterising all exceptional candidate branch data remains an open problem (see \cite{Zhe06,Zhu19,WWX25}).

Various approaches have been developed to attack this problem, including dessins d'enfants, Speiser graphs, and monodromy methods; see \cite{Bar01, BP23, EKS84, Ger87, KZ96, Med84, Med90, MSS04, OP06, Pak09, P20, PP09, PP12, PP06, PP07, PP08, SX20, Th65} for further details. We refer the reader to \cite{P20, PP08} for comprehensive surveys of known results and techniques.

More recently, motivated by the study of spherical metrics with conical singularities, Zhu \cite{Zhu19} provided new exceptional branch data. Subsequently, we developed a geometric approach to the Hurwitz existence problem in \cite{Wei26}. To state our results clearly, we recall the notions of a \emph{conical singularity} and a \emph{football metric}---or simply a \emph{football}.

\begin{defn}[\cite{WZ}]\label{DF1}
Let ${\rm d}s^{2}=e^{2\psi}|{\rm d}z|^{2}$ be a conformal metric on the punctured disk $\mathbb{D}\setminus\{0\}$. The singular point $z=0$ is called a conical point with singular angle $2\pi\alpha$ ($0<\alpha\neq 1$) if and only if $\psi$ can be locally expressed as
\[
\psi(z)=(\alpha-1)\ln |z|+\rho(z)
\]
with $\rho(z)$ a smooth function on $\mathbb{D}$.
\end{defn}

Note that in Definition \ref{DF1}, when $\alpha=1$ the metric is smooth at the singularity; however, in this paper we shall still refer to such a point as a conical point of the metric.

\begin{defn}[\cite{WWX25}]
We call a spherical conical metric on $S^{2}\cong\mathbb{C}\cup\{\infty\}$ a football metric (or simply a football) if it has two equal cone angles at antipodal points. This metric can be expressed explicitly as
\[
{\rm d}s^{2}=\frac{4\alpha^{2}|z|^{2(\alpha-1)}}{(1+|z|^{2\alpha})^{2}}|{\rm d}z|^{2},
\]
and is denoted by $S^{2}_{\{\alpha,\alpha\}}$ in this paper, where $\alpha>0$ is the cone angle parameter.
\end{defn}

\begin{figure}[htbp]
\centering
\includegraphics[width=8cm]{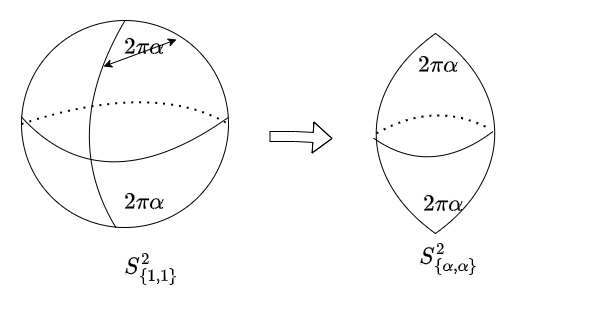}
\caption{Constructing a football $S^{2}_{\{\alpha,\alpha\}}$ from the standard football.}
\label{fig:Football}
\end{figure}

For example, starting from the standard football $S^2 = S^2_{\{1,1\}}$, one can construct a football $S^2_{\{\alpha,\alpha\}}$ by taking a bigon with angle $2\pi\alpha$ ($0<\alpha<1$) and gluing along its meridians (see Figure~\ref{fig:Football}).

\begin{thm}[\textbf{Football decomposition for rational maps}]\cite{Wei26}\label{Thm1}
Let $f: \overline{\mathbb{C}} \to \overline{\mathbb{C}}$ be a rational map of degree $\geq 2$, and let ${\rm d}s_0^2$ be the constant curvature $1$ metric on the target sphere. Then the pullback metric $f^*{\rm d}s_0^2$ is a constant curvature $1$ metric with finitely many conical singularities on the source sphere. Moreover, the space $(\overline{\mathbb{C}}, f^*{\rm d}s_0^2)$ admits a canonical decomposition: by cutting along a finite set of geodesics connecting the poles, zeros, and critical points of $f$, it can be partitioned into finitely many pieces, each isometric to a football.
\end{thm}

Based on this result, we introduce the following definition.

\begin{defn}
Let $f: \overline{\mathbb{C}} \to \overline{\mathbb{C}}$ be a rational map of degree $d\geq2$, and let ${\rm d}s_0^2$ be the constant curvature $1$ metric on the target sphere. The football decomposition of $(\overline{\mathbb{C}}, f^*{\rm d}s_0^2)$ is called the football decomposition of the rational map $f$.
\end{defn}

This paper employs the football decomposition of rational maps to investigate the realizability of a specific family of rational maps with three branch points. Recall that, for a given branch datum, the realizability of a branched covering $f \colon S^{2} \to S^{2}$ is equivalent to that of a rational map.

Let
\[
\pi_{0}=[\underbrace{2,\ldots,2}_{2s-1}],\quad
\pi_{\infty}=[\underbrace{3,\ldots,3}_{s},\underbrace{1,\ldots,1}_{s-2}],\quad
\pi_{1}=[\gamma_1,\gamma_2,\gamma_3]
\]
be three nontrivial partitions of the integer \(d=4s-2\), where \(s\geq3\). Denote $\{\pi_{0},\pi_{\infty},\pi_{1}\}$ by $\mathcal D$.
To the best of our knowledge, the realizability of $\mathcal D$ is known only in the following case, due to Boccara \cite{Bo82} and Song--Xu \cite{SX20}.

\begin{thm}\cite{Bo82,SX20}
If $\pi_{1}=[1,1,4s-4]$, then $\mathcal D$ is realizable.
\end{thm}

Turning to the remaining cases, we state our first result below.

\begin{thm}\label{MThm1}
If $\gamma_{1}=1$ and $\gamma_{3}>\gamma_{2}\geq2$, then the following statements hold.
\begin{enumerate}
\item When $s=3$, $\mathcal D$ is realizable if and only if $\pi_{1}=[1,2,7]$ or $[1,4,5]$.
\item When $s\geq4$, $\mathcal D$ is realizable.
\end{enumerate}
\end{thm}

To state our second result more clearly, we need the following definition.

If $\mathcal D$ is realizable, then there exists a rational map $f: \overline{\mathbb{C}}\to \overline{\mathbb{C}}$ with branch datum $\mathcal D$ and branch points $0,\infty,1$, where these branch points correspond to the partitions $\pi_0,\pi_\infty,\pi_1$, respectively. Let $\mathrm{d}s_0^2$ denote the standard metric on the target sphere $\overline{\mathbb{C}}$, and let $\gamma$ be its equator (see Figure~\ref{fig:Hor}). Then the preimage $f^{-1}(\gamma)$ is a geodesic curve that contains the preimages of the branch point $1$. We shall refer to $f^{-1}(\gamma)$ as a \emph{horizontal geodesic}. It is easy to see that $f^{-1}(\gamma)$ is not simple. This observation leads to the following definition.

\begin{figure}[htbp]
\centering
\includegraphics[width=3cm]{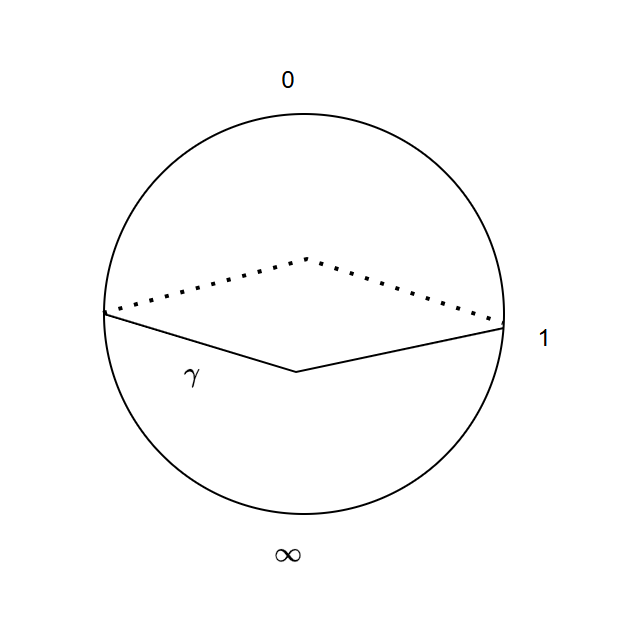}
\caption{Equator $\gamma$.}
\label{fig:Hor}
\end{figure}

\begin{defn}
We say that $f$ is of Type-A if there exists a closed simple curve $\Gamma$ of length $6\pi$ such that $\Gamma \subseteq f^{-1}(\gamma)$ and $f^{-1}(1) \subseteq \Gamma$; otherwise, $f$ is of Type-B.
\end{defn}

For example, Figure~\ref{fig:TypeA} illustrates a rational map $f$ whose branch datum is
\[
\{[2,2,2,2],[3,3,2],[3,3,2]\}.
\]
It is easy to see that $f$ is of Type-A.
\begin{figure}[htbp]
\centering
\begin{minipage}{0.35\textwidth}
    \centering
    \includegraphics[width=\linewidth]{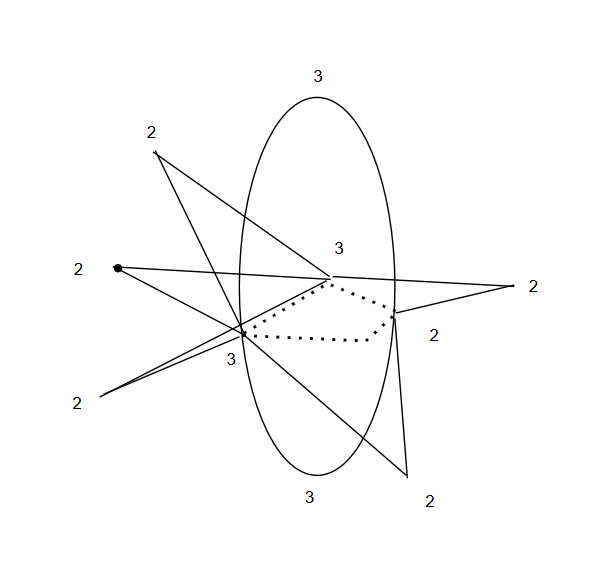}
    \caption{Figure of the map $f$.}
    \label{fig:TypeA}
\end{minipage}
\hfill
\begin{minipage}{0.35\textwidth}
    \centering
    \includegraphics[width=\linewidth]{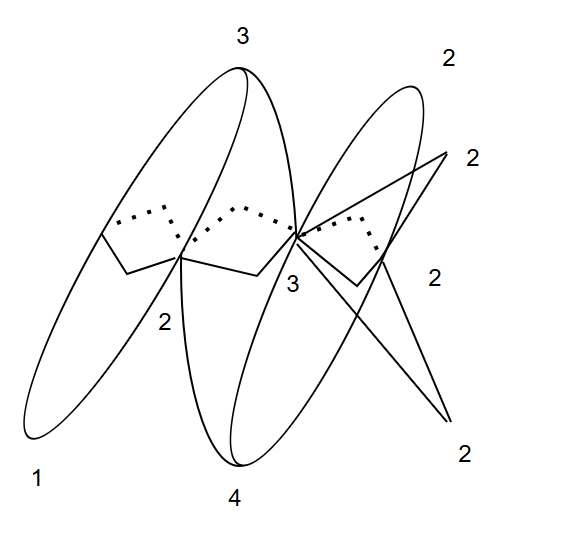}
    \caption{Figure of the map $g$.}
    \label{fig:TypeB}
\end{minipage}
\end{figure}

Figure \ref{fig:TypeB} illustrates a rational map $g$ whose branch datum is
\[
\{[3,2,2],[4,2,1],[3,2,2]\}.
\]
It is easy to see that $g$ is of Type-B.

\begin{thm}\label{MThm2}
If \(\gamma_i \ge 2\) for each \(i=1,2,3\), then \(\mathcal D\) is realizable by a rational map \(f:\overline{\mathbb{C}}\to \overline{\mathbb{C}}\) with branch points \(0,\infty,1\), corresponding to \(\pi_0,\pi_\infty,\pi_1\), respectively, if and only if one of the following statements holds.

\begin{enumerate}
\item[(A)] The map \(f\) is of Type-A and one of the following conditions holds.
\begin{enumerate}
\item[(A-1)] For every permutation \(\sigma\) of \(\{1,2,3\}\), the set \(\{\widetilde{\pi}_0,\widetilde{\pi}_\infty,\widetilde{\pi}_1\}\) is realizable, where
\[
\widetilde{\pi}_0=[\underbrace{2,\ldots,2}_{2s+1}],~
\widetilde{\pi}_\infty=[\underbrace{3,\ldots,3}_{s+1},\underbrace{1,\ldots,1}_{s-1}],~
\widetilde{\pi}_1=[\gamma_{\sigma(1)}+3,\ \gamma_{\sigma(2)}+1,\ \gamma_{\sigma(3)}].
\]
\item[(A-2)] There exists a permutation \(\tau\) of \(\{1,2,3\}\) such that the set \(\{\widehat{\pi}_0,\widehat{\pi}_\infty,\widehat{\pi}_1\}\) is realizable, where
\[
\widehat{\pi}_0=[\underbrace{2,\ldots,2}_{2s+1}],~
\widehat{\pi}_\infty=[\underbrace{3,\ldots,3}_{s+1},\underbrace{1,\ldots,1}_{s-1}],~
\widehat{\pi}_1=[\gamma_{\tau(1)}+4,\ \gamma_{\tau(2)},\ \gamma_{\tau(3)}].
\]
\end{enumerate}

\item[(B)] The map \(f\) is of Type-B and one of the following conditions holds.
\begin{enumerate}
\item[(B-1)] There exists a permutation \(\tau\) of \(\{1,2,3\}\) such that the set \(\{\widetilde{\pi}_0,\widetilde{\pi}_\infty,\widetilde{\pi}_1\}\) is realizable, where
\[
\widetilde{\pi}_0=[\underbrace{2,\ldots,2}_{2s+1}],~
\widetilde{\pi}_\infty=[\underbrace{3,\ldots,3}_{s+1},\underbrace{1,\ldots,1}_{s-1}],~
\widetilde{\pi}_1=[\gamma_{\tau(1)}+3k,\ \gamma_{\tau(2)}+k,\ \gamma_{\tau(3)}].
\]
\item[(B-2)] There exists a permutation \(\sigma\) of \(\{1,2,3\}\) such that the set \(\{\widehat{\pi}_0,\widehat{\pi}_\infty,\widehat{\pi}_1\}\) is realizable, where
\[
\widehat{\pi}_0=[\underbrace{2,\ldots,2}_{2s+1}],~
\widehat{\pi}_\infty=[\underbrace{3,\ldots,3}_{s+1},\underbrace{1,\ldots,1}_{s-1}],~
\widehat{\pi}_1=[\gamma_{\sigma(1)}+4,\ \gamma_{\sigma(2)},\ \gamma_{\sigma(3)}].
\]
\end{enumerate}
\end{enumerate}
\end{thm}

\begin{rem}
If $f$ is of Type-B, then $s\geq4$.
\end{rem}

As applications of Theorem~\ref{MThm2}, we have the following results.

\begin{cor}\label{Co1}
Let $k \ge 0$ be an integer. Then the following sets of candidate branch data are realizable:
\begin{enumerate}
\item $\displaystyle
\bigl\{
   \bigl[ \underbrace{2,\ldots,2}_{5+2k} \bigr],\
   \bigl[ \underbrace{3,\ldots,3}_{3+k},\underbrace{1,\ldots,1}_{1+k} \bigr],\
   [2+3k,\,3+k,\,5]
\bigr\}.
$
\item $\displaystyle
\bigl\{
   \bigl[ \underbrace{2,\ldots,2}_{5+2k} \bigr],\
   \bigl[ \underbrace{3,\ldots,3}_{3+k},\underbrace{1,\ldots,1}_{1+k} \bigr],\
   [2+3k,\,3,\,5+k]
\bigr\}.
$

\item $\displaystyle
\bigl\{
   \bigl[ \underbrace{2,\ldots,2}_{5+2k} \bigr],\
   \bigl[ \underbrace{3,\ldots,3}_{3+k},\underbrace{1,\ldots,1}_{1+k} \bigr],\
   [2,\,3+3k,\,5+k]
\bigr\}.
$

\item $\displaystyle
\bigl\{
   \bigl[ \underbrace{2,\ldots,2}_{5+2k} \bigr],\
   \bigl[ \underbrace{3,\ldots,3}_{3+k},\underbrace{1,\ldots,1}_{1+k} \bigr],\
   [2+k,\,3+3k,\,5]
\bigr\}.
$

\item $\displaystyle
\bigl\{
   \bigl[ \underbrace{2,\ldots,2}_{5+2k} \bigr],\
   \bigl[ \underbrace{3,\ldots,3}_{3+k},\underbrace{1,\ldots,1}_{1+k} \bigr],\
   [2+k,\,3,\,5+3k]
\bigr\}.
$

\item $\displaystyle
\bigl\{
   \bigl[ \underbrace{2,\ldots,2}_{5+2k} \bigr],\
   \bigl[ \underbrace{3,\ldots,3}_{3+k},\underbrace{1,\ldots,1}_{1+k} \bigr],\
   [2,\,3+k,\,5+3k]
\bigr\}.
$

\item $\displaystyle
\bigl\{
   \bigl[ \underbrace{2,\ldots,2}_{5+2k} \bigr],\
   \bigl[ \underbrace{3,\ldots,3}_{3+k},\underbrace{1,\ldots,1}_{1+k} \bigr],\
   [2,\,3,\,5+4k]
\bigr\}.
$
\end{enumerate}
\end{cor}

\begin{proof}
The proof of Theorem~\ref{MThm2} establishes the realizability of the base candidate branch datum
\[
\bigl\{
 [2,2,2,2,2],\
  [3,3,3,1],\
  [2,3,5]
\bigr\}.
\]
Items (1)--(6) then follow immediately from part (A) of Theorem~\ref{MThm2}. Item (7) follows from part (ii) of Theorem~\ref{MThm2}, as also shown in the proof of that theorem.
\end{proof}

\begin{cor}\label{Co2}
Let $k \ge 0$ be an integer. Then the following sets of candidate branch data are realizable:
\begin{enumerate}
\item $\displaystyle
\bigl\{
   \bigl[ \underbrace{2,\ldots,2}_{7+2k} \bigr],\
   \bigl[ \underbrace{3,\ldots,3}_{4+k},\underbrace{1,\ldots,1}_{2+k} \bigr],\
   [4+3k,\,6+k,\,4]
\bigr\}.
$
\item $\displaystyle
\bigl\{
   \bigl[ \underbrace{2,\ldots,2}_{7+2k} \bigr],\
   \bigl[ \underbrace{3,\ldots,3}_{4+k},\underbrace{1,\ldots,1}_{2+k} \bigr],\
   [4+k,\,6+3k,\,4]
\bigr\}.
$

\item $\displaystyle
\bigl\{
   \bigl[ \underbrace{2,\ldots,2}_{7+2k} \bigr],\
   \bigl[ \underbrace{3,\ldots,3}_{4+k},\underbrace{1,\ldots,1}_{2+k} \bigr],\
   [4+4k,\,6,\,4]
\bigr\}.
$
\end{enumerate}
\end{cor}
\begin{proof}
The proof of Theorem~\ref{MThm2} establishes the realizability of the base candidate branch datum
\[
\bigl\{
 [2,2,2,2,2,2],\
  [3,3,3,3,1,1],\
  [4,6,4]
\bigr\}.
\]
Consequently, items (1)--(3) follow immediately from Theorem~\ref{MThm2}(B) and Figure~\ref{fig:TypeB1}.
\end{proof}

As an application of Corollaries~\ref{Co1} and~\ref{Co2}, one can obtain all exceptional data of the form
\[
\bigl\{ \bigl[ \underbrace{2,\ldots,2}_{5+2k} \bigr],\,
\bigl[ \underbrace{3,\ldots,3}_{3+k},\underbrace{1,\ldots,1}_{1+k} \bigr],\,
[\gamma_1,\gamma_2,\gamma_3] \bigr\},
\]
with \(\gamma_i \ge 2\) for each \(i=1,2,3\). For example, when \(k=1\), it is easy to see that the following branch data are precisely all exceptional data of this form:
\begin{enumerate}
\item \[
\bigl\{ \bigl[ \underbrace{2,\ldots,2}_{7} \bigr],\,
\bigl[ 3,3,3,3,1,1 \bigr],\,
[2,5,7] \bigr\}.
\]
\item \[
\bigl\{ \bigl[ \underbrace{2,\ldots,2}_{7} \bigr],\,
\bigl[ 3,3,3,3,1,1 \bigr],\,
[3,4,7] \bigr\}.
\]
\end{enumerate}

The remainder of this paper is organised as follows. Section~\ref{sec2} reviews the classical theory of \emph{dessins d'enfants} and observes that dessins d'enfants and the football decomposition constitute two complementary approaches to the Hurwitz problem. Section~\ref{sec3} is devoted to the proof of Theorem~\ref{MThm1}. Section~\ref{sec4} is devoted to the proof of Theorem~\ref{MThm2}.

\section{Rational maps with three branch points}\label{sec2}

In this section we recall two approaches---dessins d'enfants and football decomposition---for studying the realizability of rational maps with three branch points.  First, we recall some notions of graphs on a surface.

\textbf{Dessins d'enfants.}

A technique that has been successfully employed to address the Hurwitz existence problem is based on the notion of \emph{dessin d'enfant} (see \cite{AG97,PB94,SA04} and the references therein). This method applies to candidate branch data with three branch points on the sphere.

\begin{defn}\cite{P20}
A \emph{graph} $\Gamma$ on a surface $M$ is a subset of $M$ consisting of finitely many points (the \emph{vertices}) and finitely many simple arcs (the \emph{edges}), possibly closed, such that each arc has its endpoints at vertices (or its single endpoint at a vertex if it is a loop), and the interiors of any two distinct arcs are disjoint. The \emph{valence} of a vertex $v$ is the number of edges for which $v$ is an endpoint, plus twice the number of loops at $v$. Equivalently, the valence of $v$ is the number of \emph{germs} of edges incident to $v$. The graph $\Gamma$ is called \emph{bipartite} if its vertices are coloured black and white so that every edge connects vertices of opposite colours.

A \emph{complementary region} of $\Gamma$ is a connected component $R$ of $M \setminus U$, where $U$ is the interior of a regular neighbourhood $N$ of $\Gamma$ in $M$. If $\Gamma$ is bipartite, we transfer the vertex colours to $\partial R$ by pulling back the vertices of $\Gamma$ via the restriction to $\partial R$ of the natural retraction $N \to \Gamma$. On each boundary component of $R$, black and white vertices alternate; consequently, the number of black vertices equals the number of white vertices on $\partial R$, and we call this common number the \emph{length} of $R$. These definitions are well-defined up to coloured homeomorphism, independent of the choice of $N$. One may also view $R$ as the closure of a component of $M \setminus \Gamma$, though the induced map on the boundary need not be an embedding, so some vertices may contribute multiple times to the length. A \emph{dessin d'enfant} on $M$ is a bipartite graph on $M$ all of whose complementary regions are discs.
\end{defn}

\begin{thm}\cite{P20}
A candidate branch datum $\{\pi_0, \pi_\infty, \pi_1\}$ for a branched covering $S^2 \to S^2$ is realizable if and only if there exists a dessin d'enfant $\Gamma$ on $S^2$ such that the valences of its black vertices are the entries of the partition $\pi_0$, the valences of its white vertices are the entries of $\pi_\infty$, and the lengths of its complementary regions are the entries of $\pi_1$.
\end{thm}

Let $f: \overline{\mathbb{C}} \to \overline{\mathbb{C}}$ be a rational map of degree $d\geq3$ with branch points at $0,\infty,1$ and branch datum $\{\pi_0, \pi_1, \pi_2\}$, where
\[
\pi_0 = [\alpha_1, \ldots, \alpha_r],\qquad
\pi_\infty = [\beta_1, \ldots, \beta_s],\qquad
\pi_1 = [\gamma_1, \ldots, \gamma_t],
\]
correspond to the points $0,\infty$ and $1$, respectively, and satisfy $r+s+t=d+2$.

Let $\mathrm{d}s_0^2$ be the standard metric on the target sphere $\overline{\mathbb{C}}$. There is a unique geodesic segment of length $\pi$ (i.e. a meridian) that joins the three branch points $0$, $\infty$, and $1$; however, there are infinitely many geodesic segments of length $\pi$ joining $0$ and $\infty$ that do not pass through the third branch point $1$. Choose such a geodesic segment and denote it by $\gamma_1$ (see Figure~\ref{fig:3bp1}). Then the preimage $f^{-1}(\gamma_1)$ forms a connected bipartite graph on the source sphere $\overline{\mathbb{C}}$. The vertices of this graph are the zeros and poles of $f$, and the number of edges equals $d$.

\begin{figure}[htbp]
\centering
\includegraphics[width=5cm]{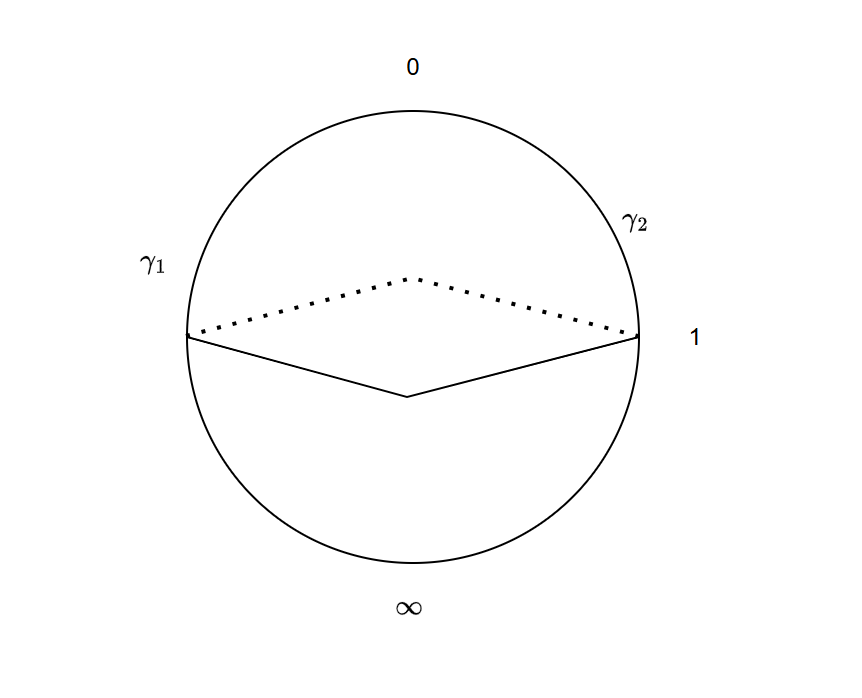}
\caption{Geodesic segments $\gamma_1$ and $\gamma_2$.}
\label{fig:3bp1}
\end{figure}

If, instead, we choose the unique geodesic segment $\gamma_2$ (see Figure~\ref{fig:3bp1}) that joins all three branch points $0$, $\infty$, and $1$, then by cutting along the preimage $f^{-1}(\gamma_2)$ including the conical singularities of $f^*\mathrm{d}s_0^2$, we obtain the football decomposition of $(\overline{\mathbb{C}}, f^*\mathrm{d}s_0^2)$. Since $0$, $\infty$, and $1$ are branch points of $f$, it is straightforward to see that each piece in this decomposition is a football whose cone angles are integer multiples of $2\pi$.

Let $n$ be the total number of pieces in the football decomposition. Since each point of order at least $2$ in $\pi_2$ lies on the boundary of the decomposition, we have
\[
n = \sum_{\gamma_j \ge 2} \gamma_j.
\]
Denote these footballs by
\[
S^2_{\{x_1, x_1\}}, \ldots, S^2_{\{x_n, x_n\}},
\]
where, for each $i$, $x_i \in \mathbb{Z}^+$ and $\sum_{i=1}^n x_i = d$. Moreover, each football $S^2_{\{x_i, x_i\}}$ contains $x_i$ edges of the bipartite graph.

Note that one of the two conical singularities of each football corresponds to a point whose order appears in the partition $\pi_0$, while the other singularity corresponds to a point whose order appears in $\pi_\infty$. Let $a_{ij}$ be the total angle contributed by the footballs that simultaneously contain the point of order $\alpha_i$ (from $\pi_0$) and the point of order $\beta_j$ (from $\pi_\infty$). Then there exists an $r\times s$ nonnegative integer matrix $A=(a_{ij})_{r\times s}$ satisfying
\[
\begin{cases}
\sum\limits_{j=1}^{s} a_{ij} = \alpha_i, & i=1,\ldots,r,\\[6pt]
\sum\limits_{i=1}^{r} a_{ij} = \beta_j, & j=1,\ldots,s,
\end{cases}
\]
and for each $i,j$, $a_{ij}$ equals the sum of some of the numbers $x_1,\dots,x_n$. Moreover, if we regard $A$ as the bipartite adjacency matrix of the graph, then that graph must be connected. Consequently, we obtain the following result.

\begin{thm}\cite{Wei26}\label{Belyi}
Let $f: \overline{\mathbb{C}} \to \overline{\mathbb{C}}$ be a rational map of degree $d\geq3$ with branch points at $0,\infty,1$ and branch datum $\{\pi_0, \pi_1, \pi_2\}$, where
\[
\pi_0 = [\alpha_1, \ldots, \alpha_r],\qquad
\pi_1 = [\beta_1, \ldots, \beta_s],\qquad
\pi_2 = [\gamma_1, \ldots, \gamma_t],
\]
correspond to the points $0,\infty$ and $1$, respectively, and satisfy $r+s+t=d+2$. Then there exist $n=\sum_{\gamma_j\ge 2}\gamma_j$ positive integers $x_1,\ldots,x_n$, an $r\times s$ nonnegative integer connected matrix $A=(a_{ij})_{r\times s}$, and integers $\delta_{ij}^k\in\{0,1\}$ for each $i,j,k$ satisfying the following system:
\begin{equation}\label{NP}
\begin{cases}
\sum\limits_{k=1}^{n} x_k = d,\\[8pt]
\sum\limits_{j=1}^{s} a_{ij} = \alpha_i, & i=1,\ldots,r,\\[6pt]
\sum\limits_{i=1}^{r} a_{ij} = \beta_j, & j=1,\ldots,s,\\[6pt]
a_{ij} = \sum\limits_{k=1}^{n} \delta_{ij}^k x_k, & i=1,\ldots,r,\ j=1,\ldots,s,\\[6pt]
\sum\limits_{j=1}^{s}\sum\limits_{i=1}^{r} \delta_{ij}^k = 1, & k=1,\ldots,n.
\end{cases}
\end{equation}
\end{thm}

\begin{rem}
 It should be pointed out that solving the system of equations (\ref{NP}) is NP-hard in Discrete Mathematics and Combinatorics.
\end{rem}

\section{Proof of Theorem \ref{MThm1}}\label{sec3}

\begin{proof}
\textbf{Necessity.}
It suffices to treat the case $s=3$.

Assume that $\mathcal D$ is realizable. Then there exists a rational map $f: \overline{\mathbb{C}} \to \overline{\mathbb{C}}$ with branch datum $\mathcal D$ and branch points $0,\infty,1$, corresponding respectively to the partitions $\pi_0,\pi_\infty,\pi_1$. Since $\gamma_1=1$ and $\gamma_i\ge 2$ for $i=2,3$, Theorem~\ref{Belyi} implies that the football decomposition of $f$ consists of eight standard footballs $S^2_{\{1,1\}}$ and one football $S^2_{\{2,2\}}$. As all parts of $\pi_0$ are equal to $2$ and all parts of $\pi_\infty$ are at most $3$, we see that the three preimages of the branch point $1$ lie on two tangent circles of length $4\pi$; in other words, $f$ is of Type-B. Hence, after removing two footballs $S^2_{\{1,1\}}$ from $f$, we obtain a block formed by six standard footballs $S^2_{\{1,1\}}$ and one football $S^2_{\{2,2\}}$ glued as in Figure~\ref{fig:New11}. This block corresponds to a rational map $g$ of degree $8$ with branch datum
\[
\bigl\{\pi_0^g=[3,3,2],\quad \pi_\infty^g=[2,2,2,2],\quad \pi_1^g=[1,2,5]\bigr\}.
\]

\begin{figure}[htbp]
\centering
\includegraphics[width=5cm]{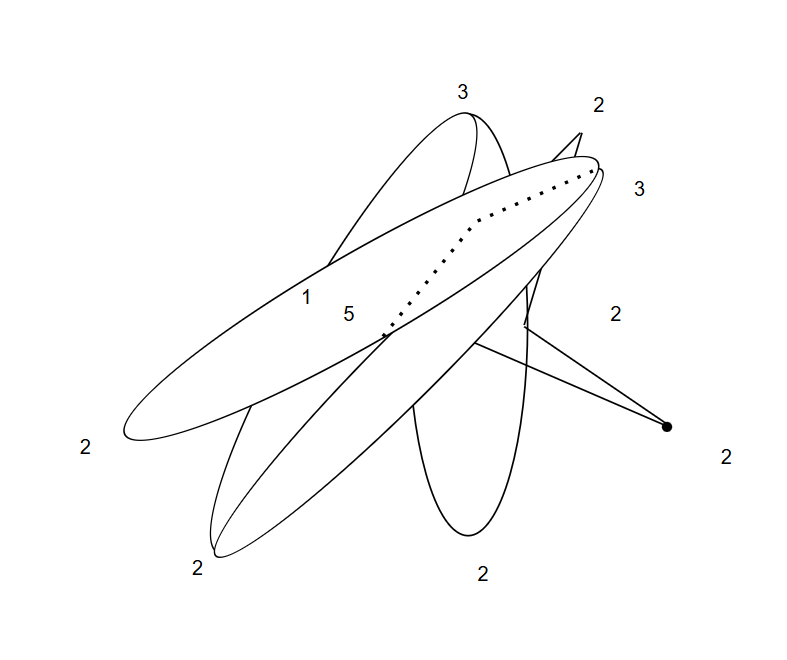}
\caption{The map $g$.}
\label{fig:New11}
\end{figure}

From the figure, the point corresponding to the part $2$ in $\pi_0^g$ (indicated by the black dot in Figure~\ref{fig:New11}) is joined by two meridians to the points corresponding to the parts $2$ and $5$ in $\pi_1^g$. Consequently, there are exactly two ways to glue two footballs $S^2_{\{1,1\}}$ onto $g$ along these two meridians, yielding rational maps $h_1$ and $h_2$ whose branch data are (see Figure~\ref{fig:New12})
\[
\bigl\{\pi_0^{h_1}=[3,3,3,1],\ \pi_1^{h_1}=[1,2,7],\ \pi_\infty^{h_1}=[2,2,2,2,2]\bigr\},
\]
and
\[
\bigl\{\pi_0^{h_2}=[3,3,3,1],\ \pi_1^{h_2}=[1,4,5],\ \pi_\infty^{h_2}=[2,2,2,2,2]\bigr\},
\]
respectively. This proves the necessity (and simultaneously establishes the sufficiency) for $s=3$.

\begin{figure}[htbp]
\centering
\includegraphics[width=9cm]{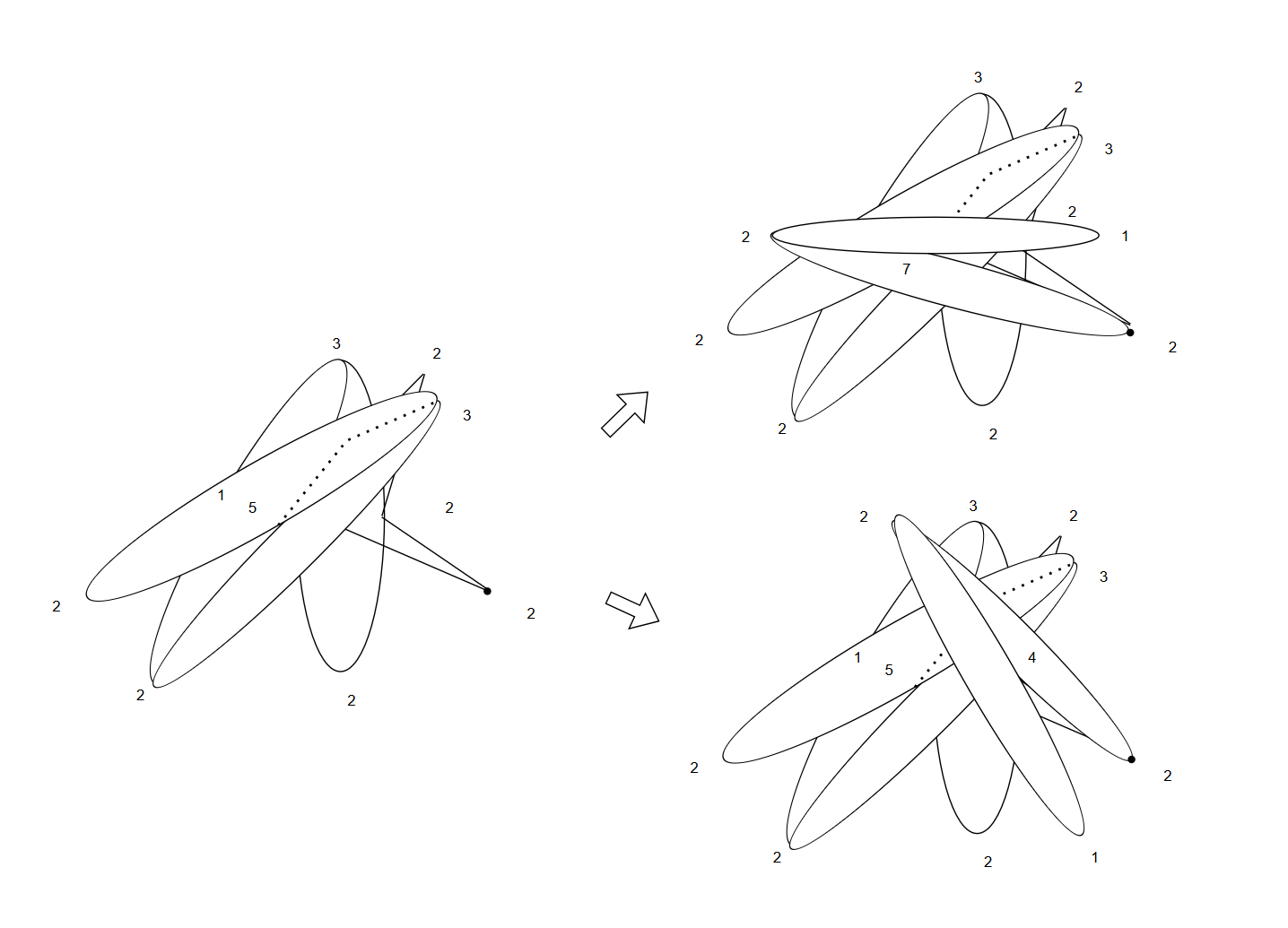}
\caption{The maps $h_1$ and $h_2$.}
\label{fig:New12}
\end{figure}

\textbf{Sufficiency.}
For \(s=3\), realizability was already established in the necessity argument. It remains to prove that \(\mathcal D\) is realizable for all \(s\ge 4\). We divide the proof into two cases.

\textbf{Case 1: \(\gamma_{2}\leq s-1\).}
We first construct a rational map \(g\) with branch datum
\[
\bigl\{[\underbrace{2,\ldots,2}_{\gamma_{2}}],\ [\underbrace{2,\ldots,2}_{\gamma_{2}}],\ [\gamma_{2},\gamma_{2}]\bigr\}.
\]
Then, by gluing \(4s-4-2\gamma_{2}\) standard footballs \(S^2_{\{1,1\}}\) and one football \(S^2_{\{2,2\}}\) onto \(g\) through a point corresponding to the order \(\gamma_{2}\), we obtain the desired rational map (see Figure~\ref{fig:New13}).

\begin{figure}[htbp]
\centering
\includegraphics[width=9cm]{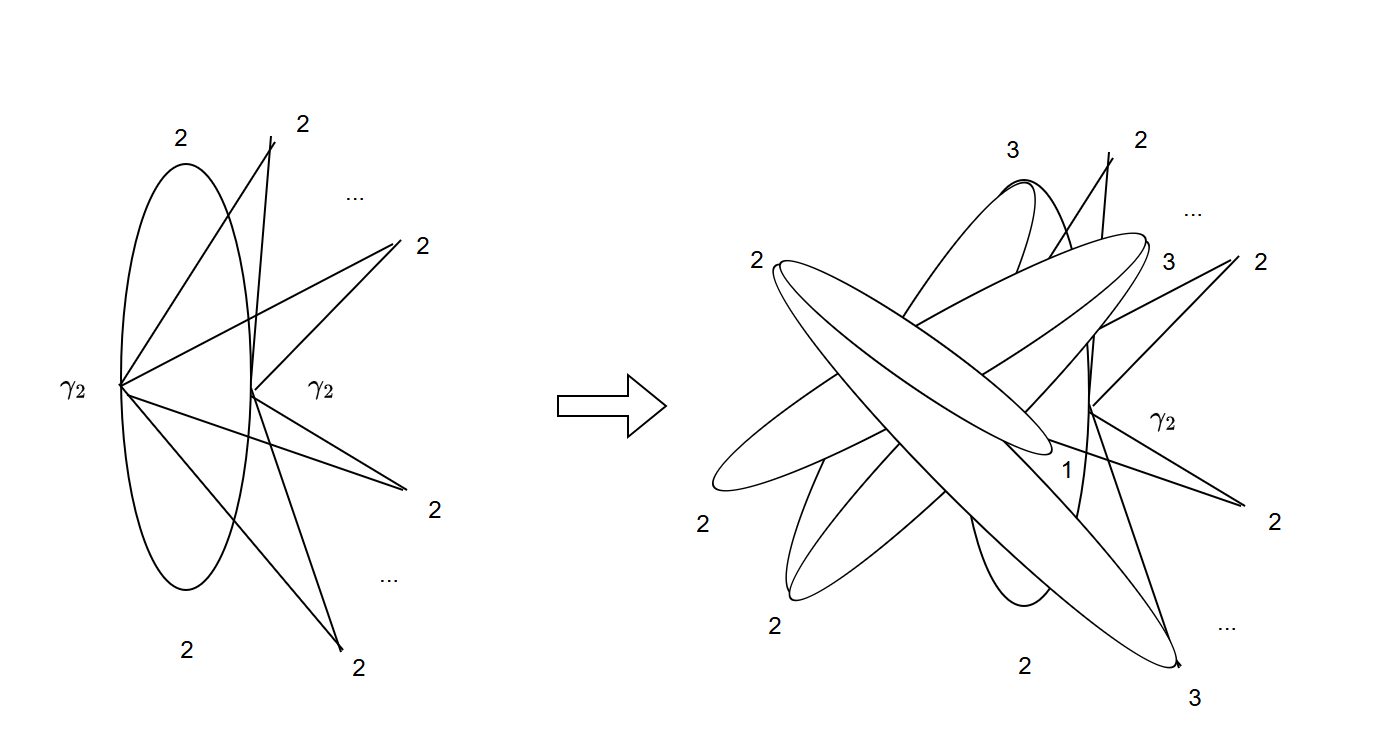}
\caption{Construction of the rational map in Case 1.}
\label{fig:New13}
\end{figure}

\textbf{Case 2: \(\gamma_{2}>s-1\).}
Since \(\gamma_{2}<\gamma_{3}\) and \(\gamma_{2}+\gamma_{3}=4s-3\), we have \(\gamma_{2}\leq 2s-2\) and \(\gamma_{2}-(s-1)\leq s-1\).
If \(\gamma_{2}-(s-1)\) is even, we begin with a rational map \(g\) whose branch datum is
\[
\bigl\{[\underbrace{2,\ldots,2}_{s-1}],\ [\underbrace{2,\ldots,2}_{s-1}],\ [s-1,s-1]\bigr\}.
\]
Then, by gluing \(\gamma_{2}-(s-1)\) standard footballs \(S^2_{\{1,1\}}\) onto \(g\) through a point corresponding to the order \(s-1\), we obtain a rational map \(h\) whose branch datum is
\[
\bigl\{[\underbrace{2,\ldots,2}_{\frac{s-1+\gamma_{2}}{2}}],\ [\underbrace{3,\ldots,3}_{\frac{\gamma_{2}-(s-1)}{2}},\underbrace{2,\ldots,2}_{\frac{3(s-1)-\gamma_{2}}{2}},\underbrace{1,\ldots,1}_{\frac{\gamma_{2}-(s-1)}{2}}],\ [\gamma_{2},s-1]\bigr\}.
\]
Finally, by gluing \(3s-3-\gamma_{2}\) standard footballs \(S^2_{\{1,1\}}\) and one football \(S^2_{\{2,2\}}\) onto \(h\) through the point corresponding to the order \(s-1\), we obtain the desired rational map.

If \(\gamma_{2}-(s-1)\) is odd, then \(\gamma_{2}-s\) is an even. When  \(\gamma_{2}-s\geq 2\), in this case, we first construct a rational map \(g\) with branch datum
\[
\bigl\{[\underbrace{2,\ldots,2}_{s-1}],\ [\underbrace{2,\ldots,2}_{s-1}],\ [s-1,s-1]\bigr\}.
\]

Then, by gluing \(\gamma_{2}-s\) standard footballs \(S^2_{\{1,1\}}\) and one football $S^{2}_{\{2,2\}}$ onto \(g\) through a point corresponding to the order \(s-1\), we obtain a rational map \(h\) whose branch datum is

\[
\bigl\{[\underbrace{2,\ldots,2}_{\frac{s+\gamma_{2}}{2}}],\ [\underbrace{3,\ldots,3}_{\frac{\gamma_{2}-s+2}{2}},\underbrace{2,\ldots,2}_{\frac{3s-\gamma_{2}-2}{2}},\underbrace{1,\ldots,1}_{\frac{\gamma_{2}-s-2}{2}}],\ [1,\gamma_{2},s-1]\bigr\}.
\]
Again, as in the previous case, gluing $3s-2-\gamma_{2}$ standard footballs \(S^2_{\{1,1\}}\) yields the desired realization of \(\mathcal D\).

When  \(\gamma_{2}-s=0\), i.e., $\gamma_{2}=s$, in this case, we first construct a rational map \(g\) with branch datum
\[
\bigl\{[\underbrace{2,\ldots,2}_{s-2}],\ [\underbrace{2,\ldots,2}_{s-2}],\ [s-2,s-2]\bigr\}.
\]

Then, by gluing \(2\) standard footballs \(S^2_{\{1,1\}}\) onto \(g\) through a point corresponding to the order \(s-2\), we obtain a rational map \(h\) whose branch datum is

\[
\bigl\{[\underbrace{2,\ldots,2}_{s-1}],\ [3,\underbrace{2,\ldots,2}_{s-3},1],\ [s,s-2]\bigr\}.
\]
Again, as in the previous case, gluing $2s-2$ standard footballs \(S^2_{\{1,1\}}\) and one football $S^{2}_{\{2,2\}}$ onto \(h\) yields the desired realization of \(\mathcal D\).
\end{proof}

\section{Proof of Theorem \ref{MThm2}}\label{sec4}

\begin{proof}
\textbf{Necessity.}

Assume that $\mathcal D$ is realizable. Then there exists a rational map $f: \overline{\mathbb{C}} \to \overline{\mathbb{C}}$ with branch datum $\mathcal D$ and branch points $0,\infty,1$, corresponding respectively to the partitions $\pi_0,\pi_\infty,\pi_1$. Since $\gamma_i \ge 2$ for each $i=1,2,3$, Theorem~\ref{Belyi} implies that the football decomposition of $f$ consists of $4s-2$ standard footballs $S^2_{\{1,1\}}$. Moreover, because all entries of $\pi_0$ are equal to $2$ and those of $\pi_\infty$ are either $3$ or $1$, the three preimages of the branch point $1$ lie either on a single circle of length $6\pi$ (Type-A) or on two tangent circles of length $4\pi$ (Type-B).

We first treat the Type-A case. After removing some standard footballs from $f$, we obtain a block formed by eight standard footballs $S^2_{\{1,1\}}$ glued as shown in Figure~\ref{fig:New31}. This block corresponds to a rational map $g$ of degree $8$ with branch datum
\[
\bigl\{\pi_0^g=[3,3,2],\quad \pi_1^g=[3,3,2],\quad \pi_\infty^g=[2,2,2,2]\bigr\}.
\]

\begin{figure}[htbp]
\centering
\begin{minipage}{0.3\textwidth}
    \centering
    \includegraphics[width=\linewidth]{New31}
    \caption{The map $g$.}
    \label{fig:New31}
\end{minipage}
\hfill
\begin{minipage}{0.55\textwidth}
    \centering
    \includegraphics[width=\linewidth]{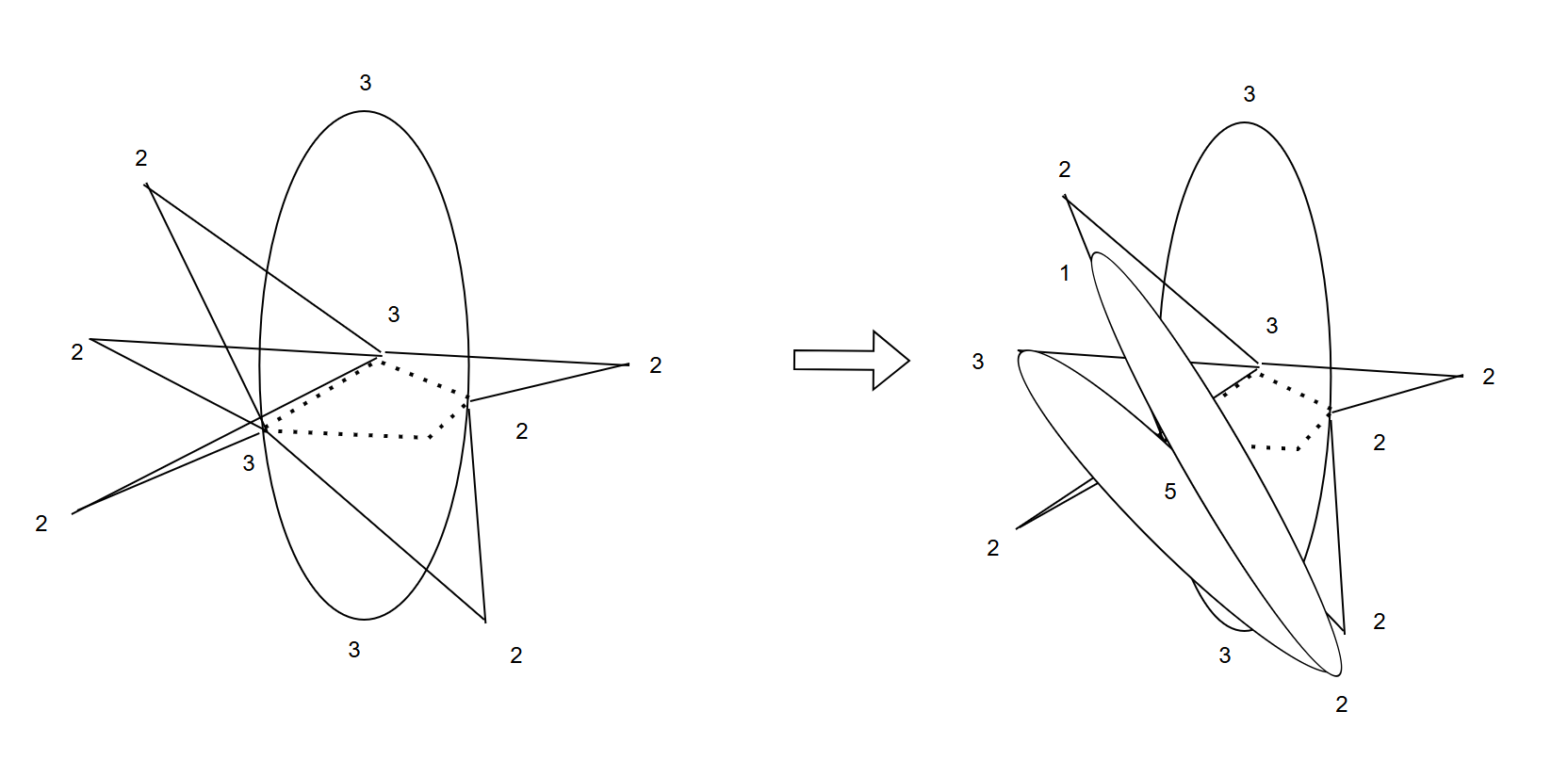}
    \caption{The map $h$.}
    \label{fig:New32}
\end{minipage}
\end{figure}

From the figure, the point corresponding to the part $2$ in $\pi_0^g$ (indicated by the black dot) is joined by two meridians to the points corresponding to the parts $3$ in $\pi_1^g$. Thus we may glue two standard footballs onto $g$ to obtain a rational map $h$ with branch datum (see Figure~\ref{fig:New32})
\[
\bigl\{\pi_0^h=[3,3,3,1],\quad \pi_1^h=[3,5,2],\quad \pi_\infty^h=[2,2,2,2,2]\bigr\}.
\]

Note that the point corresponding to the part $1$ in $\pi_0^h$ is connected by a meridian to the (unique) point corresponding to the part $5$ in $\pi_1^h$. Consequently, there are two ways to obtain a rational map whose branch datum contains the partitions $[3,3,3,3,1]$ and $[2,2,2,2,2]$ by gluing four standard footballs onto $h$.

The first construction proceeds as follows. Glue two standard footballs directly along the geodesic connecting the point for the part $1$ in $\pi_0^h$ to the point for the part $5$ in $\pi_1^h$. This yields a rational map with branch datum
\[
\bigl\{[3,3,3,3],\ [3,7,2],\ [2,2,2,2,2,1,1]\bigr\}.
\]

\begin{figure}[htbp]
\centering
\includegraphics[width=12cm]{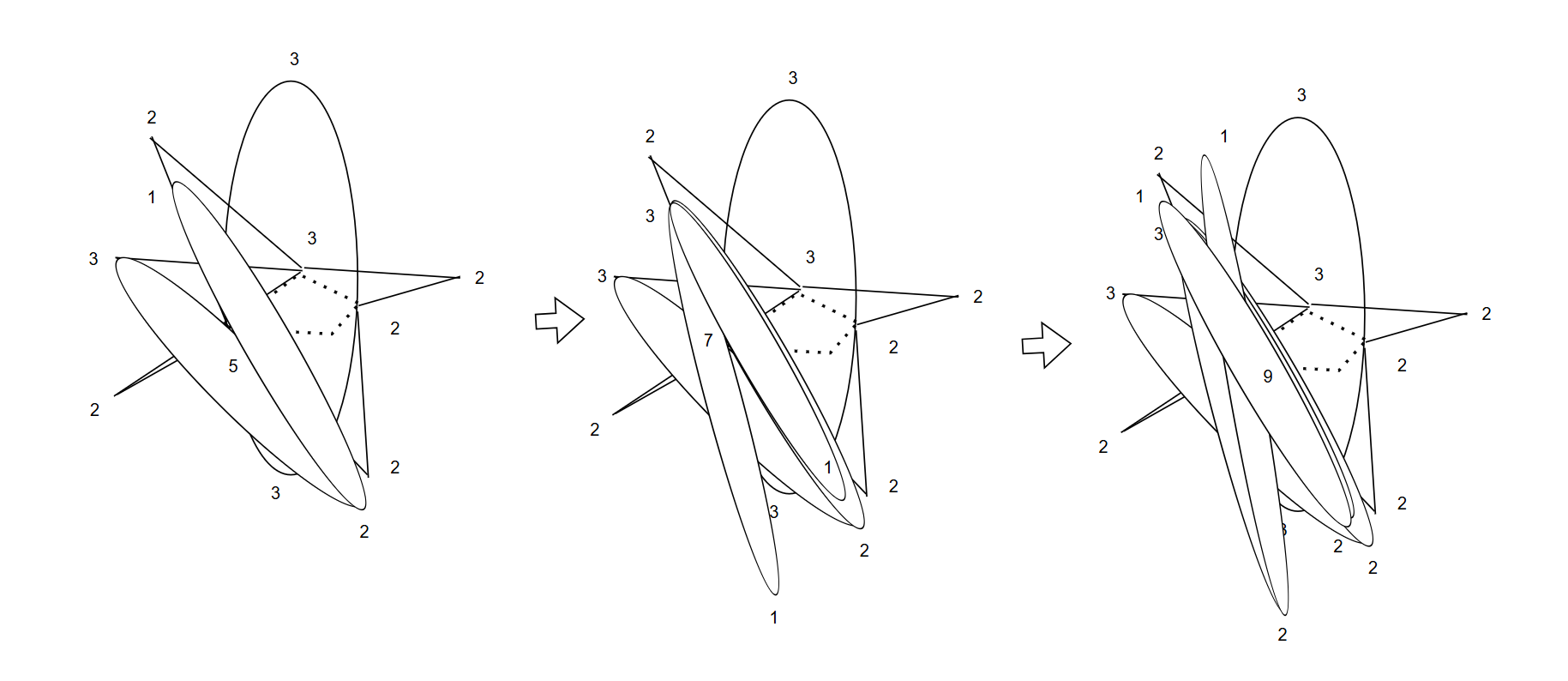}
\caption{Constructions of the desired map.}
\label{fig:New35}
\end{figure}

Then, gluing two additional standard footballs onto this new map as shown in Figure~\ref{fig:New35} produces a rational map with branch datum
\[
\bigl\{[3,3,3,3,1,1],\ [3,9,2],\ [2,2,2,2,2,2,2]\bigr\}.
\]

The second construction is as follows. Choose two points in the preimage of $1$ (for instance, the points corresponding to the parts $3$ and $5$ in the partition $[3,5,2]$) and glue two standard footballs as indicated in the left part of Figure~\ref{fig:New34}. This gives a rational map with branch datum
\[
\bigl\{[3,3,3,2,1],\ [4,6,2],\ [2,2,2,2,2,2]\bigr\}.
\]

\begin{figure}[htbp]
\centering
\includegraphics[width=8cm]{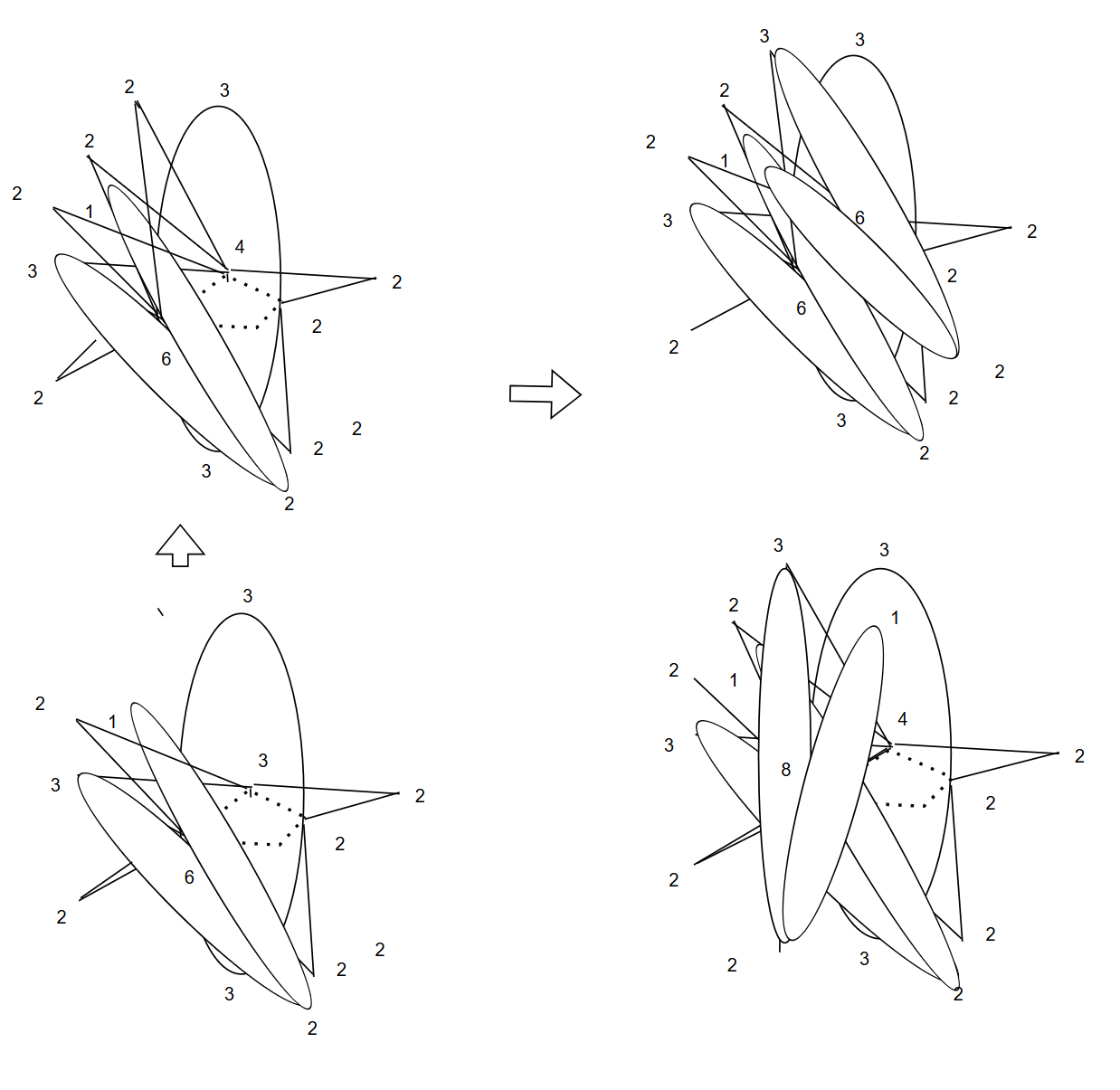}
\caption{The gluing constructions.}
\label{fig:New34}
\end{figure}

Gluing two more standard footballs onto this new map as shown in the right part of Figure~\ref{fig:New34} yields a rational map whose branch datum is either
\[
\bigl\{[3,3,3,3,1,1],\ [6,6,2],\ [2,2,2,2,2,2,2]\bigr\}\quad\text{(top)}
\]
or
\[
\bigl\{[3,3,3,3,1,1],\ [4,8,2],\ [2,2,2,2,2,2,2]\bigr\}\quad\text{(bottom)}.
\]

Iterating this gluing procedure establishes part (A) of Theorem~\ref{MThm2}.

We now turn to the Type-B case. After deleting some standard footballs from $f$, we obtain a block consisting of $14$ standard footballs $S^2_{\{1,1\}}$ glued as shown in the right part of Figure~\ref{fig:TypeB1}. This block corresponds to a rational map $g$ of degree $14$ with branch datum
\[
\bigl\{\pi_0^g=[2,2,2,2,2,2,2],\quad \pi_1^g=[4,4,6],\quad \pi_\infty^g=[3,3,3,3,1,1]\bigr\}.
\]

\begin{figure}[htbp]
\centering
\includegraphics[width=9cm]{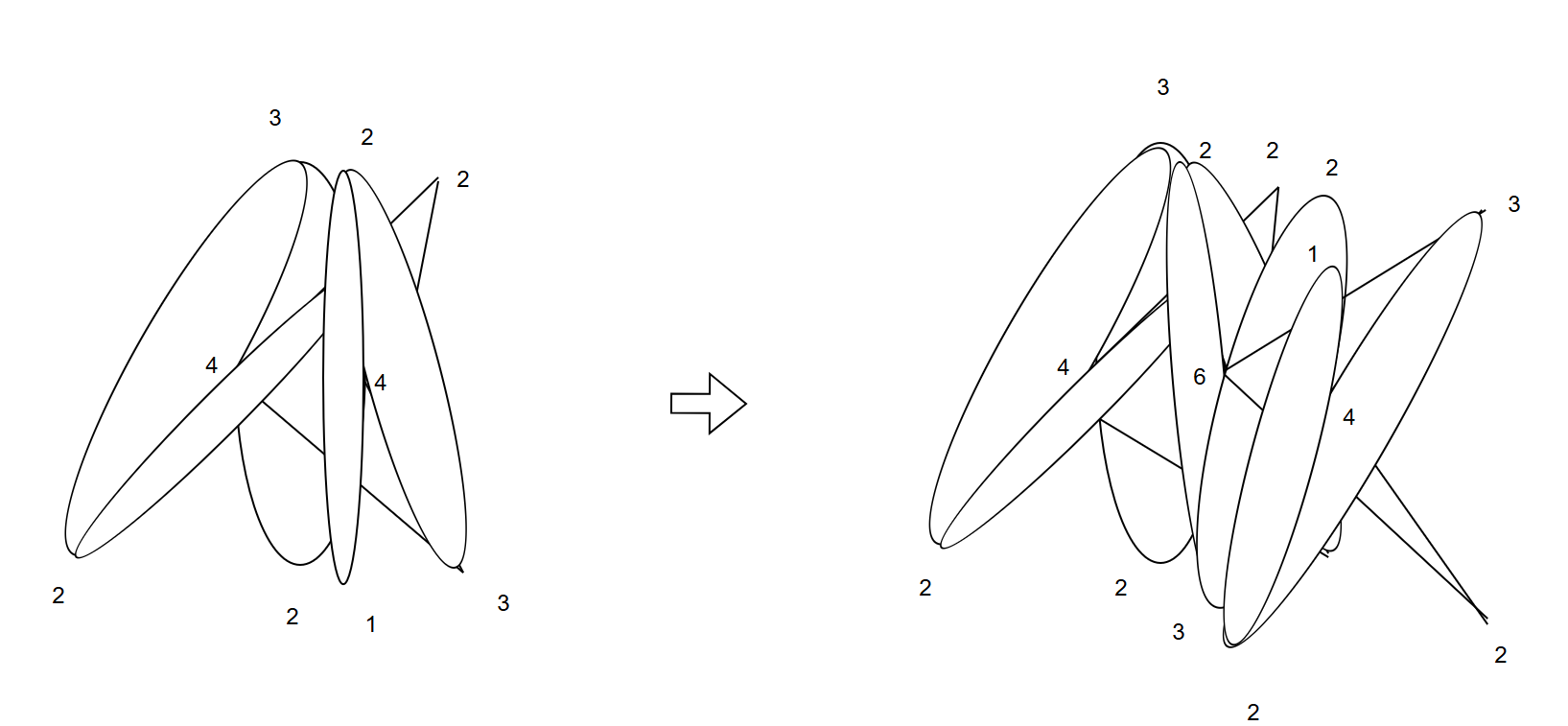}
\caption{The map $g$ in the Type-B case.}
\label{fig:TypeB1}
\end{figure}

From the figure, it is straightforward to see that there are four possible ways to glue four standard footballs onto $g$ so that the resulting map has branch datum containing the partitions $[\underbrace{3,\ldots,3}_{5},1,1,1]$ and $[\underbrace{2,\ldots,2}_{9}]$. The remainder of the proof of part (B) is analogous to that of part (A) and is therefore omitted.

\textbf{Sufficiency.}

The gluing constructions described above are reversible: starting from the given branch datum, one can successively decompose and then reconstruct the corresponding football decomposition, thereby obtaining the desired rational map. Hence sufficiency follows.
\end{proof}

\begin{rem}
Note that Theorems~\ref{MThm1} and \ref{MThm2} can also be proved via the method of dessins d'enfants; we leave the details to the interested reader.
\end{rem}


\end{document}